\documentclass[12pt]{article}
\usepackage{booktabs}
\usepackage{caption}
\usepackage{mathrsfs}
\usepackage{amsmath}
\usepackage{amsfonts,amsthm,amssymb,mathrsfs,bbding}
\usepackage{float}
\usepackage{graphics,epstopdf}
\usepackage{graphics,multicol}
\usepackage{graphicx}
\usepackage{color}
\usepackage{enumerate}
\usepackage{caption}
\usepackage{rotating}
\usepackage{lscape}
\usepackage{longtable}
\allowdisplaybreaks[4]
\usepackage{tabularx}
\usepackage[colorlinks=true,anchorcolor=blue,filecolor=blue,linkcolor=blue,urlcolor=blue,citecolor=blue]{hyperref}
\usepackage{extarrows}
\usepackage{cite}
\usepackage{latexsym,bm}
\usepackage{mathtools}
\evensidemargin=\oddsidemargin\topmargin=-1.5cm

\newtheorem{thm}{Theorem}
\newtheorem{prob}{Problem}

\newtheorem{lem}{Lemma}[section]

\newtheorem{claim}{Claim}[section]

\addtocounter{section}{0}
\begin{document}
\title{Counting substructures and eigenvalues I: triangles}

\author{Bo Ning\thanks{College of Computer Science, Nankai University, Tianjin 300350, P.R.
China. Email: bo.ning@nankai.edu.cn. Partially supported
by the NSFC grant (No.\ 11971346).}~~~
~~~Mingqing Zhai\thanks{Corresponding author. School of Mathematics and Finance, Chuzhou
University, Chuzhou, Anhui 239012, P.R. China. Email: mqzhai@chzu.edu.cn. Partially supported by the NSFC grant
(No. 12171066) and APNSF (No. 2108085MA13).}}

\maketitle
{\flushleft\large\bf Abstract:}
Motivated by the counting results for color-critical subgraphs by Mubayi
[Adv. Math., 2010], we study the phenomenon behind Mubayi's theorem
from a spectral perspective and start up this problem with the fundamental case of triangles.
We prove tight bounds on the number of copies of triangles in a graph with a prescribed number of
vertices and edges and spectral radius. Let $n$ and $m$ be the order and size
of a graph. Our results extend those of Nosal, who proved there is one
triangle if the spectral radius is more than $\sqrt{m}$, and of Rademacher, who proved there are
at least $\lfloor\frac{n}{2}\rfloor$ triangles if the number of edges is more than
that of 2-partite Tur\'an graph. These results, together with two spectral inequalities
due to Bollob\'as and Nikiforov, can be seen as a solution to the case of triangles
of a problem of finding spectral versions of Mubayi's theorem. In addition,
we give a short proof of the following inequality due to Bollob\'as and Nikiforov
[J. Combin. Theory Ser. B, 2007]: $t(G)\geq \frac{\lambda(G)(\lambda^2(G)-m)}{3}$
and characterize the extremal graphs.
Some problems are proposed in the end.

\begin{flushleft}
\end{flushleft}

\section{Introduction}

The fundamental Mantel's theorem \cite{M07} in 1907
determines the maximum number of edges in a triangle-free
graph of order $n$, that is, $\lfloor\frac{n^2}{4}\rfloor$;
and the extremal graph is
$T_{n,2}\cong K_{\lfloor\frac{n}{2}\rfloor,\lceil\frac{n}{2}\rceil}$.
Rademacher (see \cite{E55}) in 1941 showed that an $n$-vertex graph on at least
$e(T_{n,2})+1$ edges contains at least $\lfloor \frac{n}{2}\rfloor$ triangles.
Later, Erd\H{o}s \cite{E62,E62-2} proved that if $k<cn$ for some small constant $c$,
then $\lfloor\frac{n^2}{4}\rfloor+k$ edges guarantee at least
$k\lfloor\frac{n}{2}\rfloor$ triangles. Furthermore, Erd\H{o}s
also conjectured the same to be true for $k<\frac{n}{2}$,
which was finally proved by Lov\'{a}sz and Simonovits \cite{LS83}.
For more results, we refer the reader to the historical comments
and literatures therein (see Chapter 6 of \cite{B78}).

Mubayi \cite{M10} extended above theorems to the class of color-critical graphs,
that is, graphs whose chromatic number can be decreased by removing an edge.
Throughout this paper, we denote by $T_{n,k}$ the $k$-partite Tur\'an graph
and $t_{n,k}=e(T_{n,k})$.
\begin{thm}[Mubayi \cite{M10}]\label{Thm:Mubayi}
Let $k\geq 2$ and $F$ be a color-critical graph with chromatic number
$\chi(F)=k+1$. There exists $\delta=\delta_F>0$ such that if $n$ is
sufficiently large and $1\leq s:=e(G)-t_{n,k} <\delta n$,
then every $n$-vertex graph with more than $t_{n,k}$ edges contains at
least $c(n,F)\cdot (e(G)-t_{n,k})$ copies of $F$, where $c(n,F)$ is
the minimum
number of copies of $F$ in the graph obtained from $T_{n,k}$ by adding one edge.
\end{thm}

Mubayi's result also extended the following theorem of Simonovits
in the sense of counting color-critical subgraphs.

\begin{thm}[Simonovits \cite{S68}]
Let $k\geq 2$ and $F$ be a color-critical graph with $\chi(F)=k+1$. Then
$ex(n,F)=t_{n,k}$ for all large enough $n\geq n_0(F)$; moreover,
$T_{n,k}$ is the unique extremal graph.
\end{thm}

Nikiforov \cite{N09-2} proved a spectral version of the color-critical theorem
as follows, which is stronger than Simonovits' theorem, as shown by Zhai
and Lin \cite{ZL21}.

\begin{thm}[Nikiforov \cite{N09-2}]\label{nikicolor}
Let $k\geq 2$ and $F$ be a color-critical graph with $\chi(F)=k+1$.
Then there exists a positive integer $n_0(F)$ such that
if $n\geq n_0(F)$ then $T_{n,k}$ is
the only extremal graph which attains the maximum
spectral radius and contains no $F$.
\end{thm}
Recently, Zhai and Lin \cite{ZL21} refined Nikiforov's
spectral color critical theorem for two cases that the
color-critical subgraph is a book or theta graph.

Our original motivation of this paper is to study the
phenomenon behind Mubayi's theorem from a spectral
perspective. The central problem is the following.

\begin{prob}\label{Prob:1}
(i) (The general case) Find a spectral version of Mubayi's
result.\footnote{One had better give an
estimate or a formula of spectral correspondence of $c(n,F)$.}\\
(ii) (The critical case) For $s=1$ (where $s$ is defined as in
Theorem \ref{Thm:Mubayi}), find the tight spectral versions
of Mubayi's result when $F$ is some particular color-critical
subgraph, such as triangle, clique, book, odd cycle
or odd wheel, etc.
\end{prob}

One can find that, unlike the edge version, the general case
in Problem \ref{Prob:1} cannot imply the critical case
directly. As a starting point, we mainly focus on the fundamental case of triangles
for Problem \ref{Prob:1}.

The study of eigenvalue conditions for triangles has a rich history.
Let $G$ be a simple and undirected graph with $m$ edges and $t(G)$
the number of triangles in $G$. Let $\lambda(G)$ be the spectral
radius of $G$, which is defined to be the maximum of modulus of
eigenvalues of adjacency matrix $A(G)$. The eigenvalues of $G$
are arranged as $\lambda(G):=\lambda_1(G)\ge\lambda_2(G)\ge \cdots\ge \lambda_n(G)$.
If there is no danger of ambiguity, we drop the notation $G$.
A classic result due to Nosal \cite{N70} states that $t(G)\geq 1$
if $\lambda(G)>\sqrt{m}$, which is called a spectral Mantel's theorem.
Nikiforov \cite{N09} extended Nosal's theorem to that $t(G)\geq 1$ if
$\lambda(G)\geq \sqrt{m}$ unless $G$ is a complete bipartite graph
(possibly with some isolated vertices). Strengthening a conjecture
due to Edwards and Elphick \cite{EE83}, Bollob\'as and Nikiforov \cite{BN07}
conjectured that $\lambda^2_1+\lambda^2_2\leq 2m(1-\frac1{r})$
for a $K_{r+1}$-free graph on at least $r+1$ vertices and $m$ edges,
This conjecture was confirmed by Lin, Ning and Wu \cite{LNW21} for
triangle-free graphs, i.e., the case of $r=2$. Furthermore, Lin
et al. \cite{LNW21} proved that every non-bipartite graph on $m$ edges
contains a triangle if $\lambda(G)\geq \sqrt{m-1}$ unless $G$ is
a $C_5$ (possibly together with some isolated vertices).
Only very recently, Zhai and Shu \cite{ZS21} extended it as follows:
If $\lambda(G)\geq \lambda(SK_{2,\frac{m-1}{2}})$ then $t(G)\geq 1$, unless
$G\cong SK_{2,\frac{m-1}{2}}$ where $SK_{2,\frac{m-1}{2}}$ is a subdivision
on one edge of $K_{2,\frac{m-1}{2}}$. For other extensions
of Nosal's theorem, see \cite{G96,ZL21,N21,ELW21}.

In 2007, Bollob\'as and Nikiforov \cite{BN07} proved a number of relations
between the number of cliques of a graph $G$ and $\lambda(G)$.
 The important triangle case can be written
as follows.

\begin{thm}[{\rm Bollob\'as and Nikiforov \cite[$r=2$~in~Theorem~2]{BN07}}]\label{Thm:SpecSize}
$$t(G)\geq \frac{\lambda(G)(\lambda^2(G)-m)}{3}.$$
\end{thm}
\begin{thm}[{\rm Bollob\'as and Nikiforov \cite[$r=2$~in~Theorem~1]{BN07}}]\label{Thm:Specorder}
$$t(G)\geq \frac{n^2}{12}\cdot \left(\lambda(G)-\frac{n}{2}\right).$$
\end{thm}

Theorems \ref{Thm:SpecSize} and \ref{Thm:Specorder} provide
two powerful results on counting triangles.
A combination of these two inequalities can be seen as a
partial solution to Problem 1 (i) for the case of $F=C_3$.
Motivated by the previous work \cite{BN07,N09-2,M10,ZL21},
we concern the similar problem on counting triangles
when $\lambda(G)$ is close to $\sqrt{m}$ or $\frac n2.$
More precisely,
we shall give a solution to Problem 1 (ii) when $F=C_3$.

Indeed, we improve Nosal's theorem by a sharp
counting result on $t(G)$ as follows.

\begin{thm}\label{Thm:spectraltriangle:counting-size}
Let $G$ be a graph with $m$ edges. If $\lambda(G)\ge\sqrt{m}$ then
$t(G)\ge \lfloor\frac{\sqrt{m}-1}2\rfloor,$
unless $G$ is a complete bipartite graph (possibly with isolated vertices).
\end{thm}

\vspace{2mm}
\noindent
\textbf{Remark.}
Let $b\le 4(a+1)$, $a\ge 1$ and
$K_{a,b}^+$ be the graph obtained from $K_{a,b}$ by
adding an edge to the color set of size $b$.
Then $e(K_{a,b}^+)=ab+1$ and $t(K_{a,b}^+)=a$.
It is easy to check $\lambda(K_{a,b}^+)\ge \sqrt{e(K_{a,b}^+)}$ for $b\le 4(a+1)$.
Obviously, we have $e(K_{a,b}^+)=ab+1\le 4a(a+1)+1,$
and so $t(K_{a,b}^+)=a\ge \frac12\left(\sqrt{e(K_{a,b}^+)}-1\right)$.
This implies that the lower bound in Theorem \ref{Thm:spectraltriangle:counting-size} is best possible.

With the method similar to the one proving Theorem \ref{Thm:spectraltriangle:counting-size},
we shall present a short new proof of Theorem \ref{Thm:SpecSize}.
Very recently, Theorem \ref{Thm:SpecSize} was further improved by
Nikiforov \cite{N21} to $\lambda^3-\lambda\cdot m+c\lambda\cdot t^{''}\leq 3t$
for connected non-bipartite graphs, where
$t^{''}:=\sum_{u\in V(G)}t^{''}(u)=\sum_{u\in V(G)}|\{vw\in E(G): u\in\overline{N}(v)\cap \overline{N}(w)\}|$,
as a powerful tool to solve an open problem by Zhai et al.
(see \cite[Conjecture~5.2]{ZLS21}). The original inequality
(Theorem \ref{Thm:SpecSize}) is also used as a tool for
obtaining a spectral version of extremal number of friendship
graphs \cite{EFGG95}.

Moreover, we also present
a strengthening of Nikiforov's spectral color critical
theorem on triangles (see Theorem \ref{nikicolor}),
and this result can be viewed as a spectral version
of Rademacher's theorem.

\begin{thm}\label{spectraltriangle:counting-order}
Let $G$ be a graph on $n$ vertices. If $\lambda(G)\ge
\sqrt{\lfloor\frac{n^2}4\rfloor}$, then $t(G)\ge \lfloor\frac{n}{2}\rfloor-1$ unless
$G$ is the bipartite Tur\'{a}n graph $T_{n,2}$.
\end{thm}

\noindent
\textbf{Remark.} When $n$ is even, we denote
$K_{\frac{n}{2}+1,\frac{n}{2}-1}^+$ by the graph obtained from
$K_{\frac{n}{2}+1,\frac{n}{2}-1}$ by adding an edge to the colorable
set of size $\frac{n}{2}+1$. One can verify
$\lambda(K_{\frac{n}{2}+1,\frac{n}{2}-1}^+)>\lambda(T_{n,2})=\frac{n}{2}$,
and the number of triangles in $K_{\frac{n}{2}+1,\frac{n}{2}-1}^+$
is exactly $\frac{n}{2}-1$. This implies that the lower bound in
Theorem \ref{spectraltriangle:counting-order} is best possible.

Some notations involved in this paper are introduced.
Let $G$ be a graph with vertex set $V(G)$ and edge set $E(G)$.
For a vertex $v\in V(G)$ (whether $v\in S$ or not),
let $N_G(v)$ (resp. $N_S(v)$) be the set of neighbors in $G$ (resp. in $S$),
and $d_G(v)=|N_G(v)|$ (resp. $d_S(v)=|N_S(v)|$). Specially, set $N_G[v]=N_G(v)\cup \{v\}$.
Let $G[S]$ be the subgraph of $G$ induced by $S$.

In the next section, we present proofs of Theorems \ref{Thm:SpecSize}
and \ref{Thm:spectraltriangle:counting-size}. We give a proof of
Theorem \ref{spectraltriangle:counting-order} in
Section \ref{Sec:3}.
We conclude this paper with some open problems in the
last section.

\section{Proofs of Theorems \ref{Thm:SpecSize} and \ref{Thm:spectraltriangle:counting-size}}\label{Sec:2}

In this section, we first write Bollob\'{a}s-Nikiforov inequality (Theorem \ref{Thm:SpecSize})
in a compact form (Theorem \ref{Thm:trianlgeequality}), whose proof uses two lemmas.
\begin{thm}\label{Thm:trianlgeequality}
Let $G$ be a graph on $n$ vertices and $m$ edges.
Let $\lambda_1\geq \lambda_2\geq\cdots\geq \lambda_n$
be all eigenvalues of $G$. Then
\begin{align}
t(G)&=\frac{1}{6}\sum_{i=2}^n(\lambda_1+\lambda_i)\lambda^2_i+\frac{\lambda_1(\lambda_1^2-m)}{3}\label{8}\\\
&\geq \frac{\lambda(G)(\lambda^2(G)-m)}{3}\label{9}.
\end{align}
In particular, Eq. in (\ref{9}) holds if and only if $G$
is a complete bipartite graph (possibly with some isolated vertices).
\end{thm}

\begin{lem}[{\rm {Theorem 3.13 in \cite[pp.~88]{CDS80}}}]\label{Lem:diameter}
Let $G$ be a connected graph. If the diameter of $G$ is $d$, then
$G$ contains at least $d+1$ distinct eigenvalues.
\end{lem}

\begin{lem}[{\rm {Theorem 3.4 in \cite[pp.~82]{CDS80}}}]\label{Lem:bipartite}
Let $G$ be a connected graph. Then $G$ is bipartite if and only $\lambda_n=-\lambda_1$.
\end{lem}

\noindent
{\bf Proof of Theorem \ref{Thm:trianlgeequality}.}
Set $\lambda:=\lambda_1=\sqrt{m+\delta}.$
Recall that
$2m=\sum_{i=1}^n \lambda^2_i.$
From two equalities above, we have
$\lambda^2_1=\sum_{i=2}^n\lambda^2_i+2\delta.$
Furthermore, we have
\begin{align}\nonumber
t(G)&=\frac{1}{6}\left(\lambda^3_1+\lambda^3_2+\ldots+\lambda^3_n\right)
=\frac{1}{6}\left(\lambda_1(\sum_{i=2}^n\lambda^2_i+2\delta)+\lambda^3_2+\ldots+\lambda^3_n\right)\\
&=\frac{1}{6}\sum_{i=2}^n\lambda^2_i(\lambda_1+\lambda_i)+\frac{1}{3}\lambda_1(\lambda^2_1-m).\nonumber
\end{align}
This proves (\ref{8}).

We shall use (\ref{8}) together with
some arguments from \cite{N17} to give a direct
and short proof of (\ref{9}).
Obviously, if $\lambda(G)<\sqrt{m}$ then there is nothing to prove.
Assume $\lambda(G)\geq\sqrt{m}$.

First suppose that $G$ is connected. By Perron-Frobenius Theorem, $\lambda_1+\lambda_i\geq 0$
holds for any integer $i\in [2,n]$. It follows that $\sum_{i=2}^n\lambda^2_i(\lambda_1+\lambda_i)\geq 0$.
From (\ref{8}), we infer that
$t(G)\geq \frac{1}{3}\lambda_1(\lambda^2_1-m).$
Furthermore, if $t(G)=\frac{1}{3}\lambda_1(\lambda^2_1-m)$ then
\begin{align}\label{ali:2ton}
\sum_{i=2}^n\lambda^2_i(\lambda_1+\lambda_i)=0.
\end{align}
From the trace formulae $\sum_{i=1}^n\lambda_i=0$ and
(\ref{ali:2ton}), we can see there exists a maximum integer
$j\in [2,n]$ such that $\lambda_j=-\lambda_1$.
Since $|\lambda_i|\leq\lambda_1$ holds
for all integers $i\in [2,n]$, $j=n$. Thus $\lambda_n=-\lambda_1$.

By Lemma \ref{Lem:bipartite}, $G$ is bipartite. From (\ref{ali:2ton})
and the fact that $\lambda_2<\lambda_1$,
all eigenvalues are $\lambda_1,\lambda_n=-\lambda_1$ and
$\lambda_2=\cdots=\lambda_{n-1}=0$. If $G$ is not complete bipartite,
then its diameter is at least 3, and by Lemma \ref{Lem:diameter},
there are at least 4 distinct eigenvalues, a contradiction.
Thus $G$ is a complete bipartite graph.
If $G$ is complete bipartite, it is easy to find (\ref{9}) holds in equality.

Now assume that $G$ is not connected. Let $H$ be a component of $G$
with $\lambda(H)=\lambda(G)$. Note that $\lambda(H)\geq \sqrt{e(H)}$,
where $e(H)$ is the number of edges in $H$.
Since Theorem \ref{Thm:trianlgeequality} is proved to be true
for the connected case, we have
$$t(G)\geq t(H)\geq \frac{1}{3}\lambda_1(\lambda^2_1-e(H))\geq \frac{1}{3}\lambda_1(\lambda^2_1-e(G)).$$
If equality holds, then $H$ is complete bipartite,
and furthermore, $e(H)=e(G)$, which implies that each of other components (if they exist)
is an isolated vertex. The converse part is obvious. The proof is complete. $\hfill\blacksquare$

\vspace{2mm}
The following lemma is known as Cauchy's interlace theorem (see \cite{CH93}),
which is a direct consequence of the Courant-Fischer-Weyl min-max principle.
A short proof of this theorem can also be found in \cite{55}.

\begin{lem}[Cauchy's Interlace Theorem]\label{Thm:CH93}
Let $A$ be a symmetric $n\times n$
matrix and $B$ be an $r\times r$ principal submatrix of $A$ for some $r<n$. If
the eigenvalues of $A$ are $\lambda_1\geq \lambda_2\geq\cdots \geq\lambda_n$
and the eigenvalues of $B$ are $\mu_1\geq \mu_2\geq\cdots\geq \mu_r$, then
$\lambda_i\geq \mu_i\geq \lambda_{i+n-r}$ for all $1\leq i\leq r$.
\end{lem}

Now we are ready to give a proof of Theorem \ref{Thm:spectraltriangle:counting-size}.

\vspace{2mm}
\noindent
{\bf Proof of Theorem \ref{Thm:spectraltriangle:counting-size}.}
If $t(G)\geq r$ for any positive integer $r$ satisfying $m\geq (2r+1)^2$,
then $t(G)\geq \big\lfloor\frac{\sqrt{m}-1}2\big\rfloor$ and the theorem holds.
In the following, let $r$ be a positive integer subject to $m\geq (2r+1)^2$,
and $G$ be a graph with minimum degree $\delta(G)\geq1$ such that $\lambda(G)\geq \sqrt{m}$ while $t(G)\leq r-1.$
To prove Theorem \ref{Thm:spectraltriangle:counting-size}, it suffices to show that
$G$ is complete and bipartite.

Now let $X=(x_1,x_2,\ldots,x_n)^T$ be the Perron vector of $G$
and $u^*\in V(G)$ such that $x_{u*}=\max_{u\in V(G)}x_u$.
We also let $U=N_G(u^*)$, $W=V(G)\setminus N_G[u^*]$ and $e(U)$ be the number of edges within $U$.

\begin{claim}\label{cl3.4}
If $e(U)=0$ then $G$ is a complete bipartite graph.
\end{claim}

\begin{proof}
We first assume that $e(W)\neq0$.
Let $e(U,W)$ be the number of edges with one endpoint in $U$ and the other in $W$.
Then \begin{eqnarray*}
\lambda^2x_{u^*}=\sum_{u\in U}\sum_{w\in N_G(u)}x_w=|U|x_{u^*}+\sum_{w\in W}d_U(w)x_w
\leq(|U|+e(U,W))x_{u^*}<mx_{u^*}.
\end{eqnarray*}
Consequently, $\lambda<\sqrt{m}$, a contradiction. Therefore, $e(W)=0$.

Now we have $e(U)=e(W)=0$. Then $G$ is triangle-free.
As mentioned in the part of introduction, a strengthening of
Nosal's theorem due to Nikiforov (see \cite{N09})
states that $\lambda\leq\sqrt{m}$ for every triangle-free graph with $m$ edges,
with equality if and only if it is a complete bipartite graph (possibly with some isolated vertices).
Recall that $\lambda(G)\geq\sqrt{m}$ and $\delta(G)\geq1$. This implies that
$\lambda(G)=\sqrt{m}$ and $G$ is complete and bipartite.
\end{proof}

By Claim \ref{cl3.4}, we may assume that $e(U)\neq0$.
In this case, $r\geq 2$ and $m\geq (2r+1)^2\geq 25$.
Now we shall determine some forbidden subgraphs of $G$.
For convenience, we first introduce a function $f(x)$.

\begin{claim}\label{cl3.1}
Let $f(x)=\left(\sqrt{m}+x\right)x^2$. If $a\le x\le b\le 0$
then $f(x)\ge \min\{f(a),f(b)\}$.
\end{claim}

\begin{proof}
Clearly, $f'(x)=x\left(3x+2\sqrt{m}\right)$.
Thus, $f(x)$ is monotonic increasing when $x\in \left(-\infty,-\frac{2}{3}\sqrt{m}\right]$,
and monotonic decreasing when $x\in \left(-\frac{2}{3}\sqrt{m},0\right]$.
Therefore, $f(x)\ge \min\{f(a),f(b)\}$.
\end{proof}

\begin{claim}\label{cl3.2}
If $\lambda_n(G)\ge -\sqrt{m-2}$ then $f(\lambda_n(G))\ge \sqrt{m}-1$ for any integer $m\ge 3$.
\end{claim}

\begin{proof}
Since every graph $G$ contains $K_2$ as an induced subgraph, by Lemma \ref{Thm:CH93},
we have $\lambda_n(G)\leq \lambda_2(K_2)=-1$.
Subsequently, by Claim \ref{cl3.1}, we have $f(\lambda_n(G))\ge \min \{f(-1), f\left(-\sqrt{m-2}\right)\}$.
Note that $f(-1)=\sqrt{m}-1$, and
it is easy to check that $f(-\sqrt{m-2})=\frac{2(m-2)}{\sqrt{m}+\sqrt{m-2}}\geq \sqrt{m}-1$ for $m\ge 3$.
This proves Claim \ref{cl3.2}.
\end{proof}

\begin{figure}[t]
\centering \setlength{\unitlength}{0.9pt}
\begin{center}
\begin{picture}(430.7,78.3)
\put(0.7,67.4){\circle*{4}} \put(50.0,67.4){\circle*{4}}
\qbezier(0.7,67.4)(25.4,67.4)(50.0,67.4) \put(0.0,29.0){\circle*{4}}
\put(49.3,29.0){\circle*{4}}
\qbezier(0.0,29.0)(24.7,29.0)(49.3,29.0)
\put(109.5,26.8){\circle*{4}} \put(86.3,26.8){\circle*{4}}
\qbezier(109.5,26.8)(97.9,26.8)(86.3,26.8)
\put(131.2,26.8){\circle*{4}}
\qbezier(109.5,26.8)(120.4,26.8)(131.2,26.8)
\put(152.3,26.8){\circle*{4}}
\qbezier(131.2,26.8)(141.7,26.8)(152.3,26.8)
\put(118.2,76.9){\circle*{4}}
\qbezier(118.2,76.9)(102.2,51.8)(86.3,26.8)
\qbezier(118.2,76.9)(113.8,51.8)(109.5,26.8)
\qbezier(118.2,76.9)(124.7,51.8)(131.2,26.8)
\qbezier(118.2,76.9)(135.2,51.8)(152.3,26.8)
\put(186.3,72.5){\circle*{4}} \put(238.5,27.6){\circle*{4}}
\put(238.5,72.5){\circle*{4}} \put(186.3,27.6){\circle*{4}}
\qbezier(238.5,72.5)(212.4,50.0)(186.3,27.6)
\qbezier(186.3,72.5)(212.4,50.0)(238.5,27.6)
\qbezier(186.3,72.5)(212.4,72.5)(238.5,72.5)
\qbezier(186.3,72.5)(186.3,50.0)(186.3,27.6)
\qbezier(238.5,72.5)(238.5,50.0)(238.5,27.6)
\qbezier(186.3,27.6)(212.4,27.6)(238.5,27.6)
\put(297.3,78.3){\circle*{4}} \put(279.9,52.2){\circle*{4}}
\qbezier(297.3,78.3)(288.6,65.3)(279.9,52.2)
\put(313.9,52.2){\circle*{4}}
\qbezier(297.3,78.3)(305.6,65.3)(313.9,52.2)
\put(298.0,26.8){\circle*{4}}
\qbezier(279.9,52.2)(288.9,39.5)(298.0,26.8)
\qbezier(313.9,52.2)(306.0,39.5)(298.0,26.8)
\qbezier(279.9,52.2)(296.9,52.2)(313.9,52.2)
\put(340.8,71.1){\circle*{4}}
\qbezier(313.9,52.2)(327.3,61.6)(340.8,71.1)
\put(341.5,59.5){\circle*{4}}
\qbezier(313.9,52.2)(327.7,55.8)(341.5,59.5)
\put(341.5,46.4){\circle*{4}}
\qbezier(313.9,52.2)(327.7,49.3)(341.5,46.4)
\put(341.5,31.9){\circle*{4}}
\qbezier(313.9,52.2)(327.7,42.1)(341.5,31.9)
\put(403.8,78.3){\circle*{4}} \put(379.2,55.1){\circle*{4}}
\qbezier(403.8,78.3)(391.5,66.7)(379.2,55.1)
\put(430.7,55.8){\circle*{4}}
\qbezier(403.8,78.3)(417.2,67.1)(430.7,55.8)
\put(389.3,27.6){\circle*{4}}
\qbezier(379.2,55.1)(384.3,41.3)(389.3,27.6)
\put(422.0,27.6){\circle*{4}}
\qbezier(430.7,55.8)(426.3,41.7)(422.0,27.6)
\qbezier(389.3,27.6)(405.6,27.6)(422.0,27.6)
\put(13.8,0.0){\makebox(0,0)[tl]{$G_{1}$}}
\put(112.4,0.7){\makebox(0,0)[tl]{$G_{2}$}}
\put(204.5,0.7){\makebox(0,0)[tl]{$G_{3}$}}
\put(301.6,0.7){\makebox(0,0)[tl]{$G_{4}$}}
\put(400.2,2.2){\makebox(0,0)[tl]{$G_{5}$}}
\end{picture}
\end{center}
\caption{Some forbidden induced subgraphs of $G$.}\label{fig01}
\end{figure}
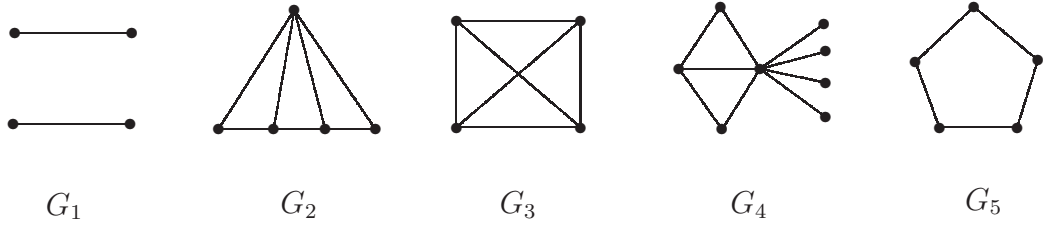

\begin{claim}\label{cl3.3}
$G$ does not contain these graphs as induced subgraphs (see Figure \ref{fig01}).
\end{claim}

\begin{proof}
Let $|G|=n$, $\lambda_j:=\lambda_j(G)$ for $j\in\{1,2,\ldots,n\}$,
and suppose to the contrary that $G$ contains $G_i$ as an induced subgraph for some $\in\{1,2,\ldots,5\}$.
To obtain a contradiction,
it suffices to show $t(G)\geq\lfloor\frac{\sqrt{m}-1}2\rfloor.$
Recall that $e(U)\neq0$, and so $t(G)\geq1$. Thus, we only need to consider
the case of $m\geq25$.

We first consider the case of $i=1$.
Lemma \ref{Thm:CH93} gives that
$\lambda_2\ge \lambda_2(G_1)=1$ and $\lambda_{n-1}\le \lambda_3(G_1)=-1.$
Since $\lambda_1^2\ge m$, we have
\begin{align}\label{la0}
\lambda_n^2=2m-\sum_{j=1}^{n-1}\lambda_j^2\le 2m-(\lambda_1^2+\lambda_2^2+\lambda_{n-1}^2)\le m-2.
\end{align}
It follows from Claim \ref{cl3.2} that
\begin{align}\label{la1}
 f(\lambda_n)\ge \sqrt{m}-1.
\end{align}
On the other hand, note that $\lambda_{n-1}^2+\lambda_{n}^2\le 2m-\lambda_1^2\le m$ and $\lambda_{n-1}^2\le \lambda_n^2$.
Then, $\lambda_{n-1}^2\leq\frac{m}{2}$, that is, $\lambda_{n-1}\geq-\sqrt{\frac{m}{2}}$.
Since $m\geq25$, by Claim \ref{cl3.1} we have
\begin{align}\label{la2}
 f(\lambda_{n-1})\ge \min \{f(-1), f(-\sqrt{m/2})\}=\sqrt{m}-1.
\end{align}
Moreover, $f(\lambda_2)\geq \sqrt{m}\cdot\lambda^2_2\ge \sqrt{m}$, as $\lambda_2\geq1$.
Combining this with (\ref{8}), (\ref{la1}) and (\ref{la2}), we have
\begin{eqnarray}\label{la-1}
t(G)\geq \frac16\left(f(\lambda_2)+f(\lambda_{n-1})+f(\lambda_n)\right)>\frac{1}{6}(3\sqrt{m}-3)=\frac{\sqrt{m}-1}{2}.
\end{eqnarray}

If $i\in\{2,4,5\}$,
then by Matlab we can see $\lambda_2(G_i)>\frac{\sqrt{3}}3$ and $\lambda_{|G_i|-1}(G_i)<-\frac43$ (see Table \ref{tab01}).
It follows from Cauchy's interlace theorem that $\lambda_2>\frac{\sqrt{3}}3$ and $\lambda_{n-1}<-\frac43$.
Observe that $\lambda_2^2+\lambda_{n-1}^2>2$.
Consequently, (\ref{la0}) holds, and so (\ref{la1}) holds.
Moreover,
we can check that $f(\lambda_{n-1})\ge \min \{f(-\frac43),f(-\sqrt{\frac m2})\}\geq\frac53\sqrt{m}-2$
for $m\geq25$.
Combining with $f(\lambda_n)\geq\sqrt{m}-1$ and $f(\lambda_2)>f(\frac{\sqrt{3}}3)>\frac{\sqrt m}3,$
we also have (\ref{la-1}).

\begin{table}[htbp]
\caption{$\lambda_2(G_i)$ and $\lambda_{|G_i|-1}(G_i)$ for $i\in\{2,4,5\}$.}\label{tab01}
\label{tab:test}
\begin{tabular*}{\textwidth}{@{}@{\extracolsep{\fill}}cccc@{}}
\toprule
& $G_2$  & $G_4$  & $G_5$ \\
 \midrule
$\lambda_2$        & 0.6180 & 0.7660 &  0.6180 \\
$\lambda_{|G_i|-1}$ & -1.4728 &  -1.3807 & -1.6180 \\
 \bottomrule
 \end{tabular*}
 \end{table}

If $i=3$, then by Cauchy's interlace theorem,
$\lambda_n\leq\lambda_{n-1}\le \lambda_{n-2}\leq \lambda_2(K_4)=-1,$
and hence (\ref{la0}-\ref{la1}) hold.
Note that $\lambda_{n-2}\geq \lambda_{n-1}\geq-\sqrt{\frac{m}2}.$
Thus, (\ref{la2}) also holds, and similarly,
$f(\lambda_{n-2})\ge \min \{f(-1), f(-\sqrt{\frac m2})\}\ge \sqrt{m}-1.$
Therefore,
\begin{align*}
 t(G)\geq \frac16(f(\lambda_{n-2})+f(\lambda_{n-1})+f(\lambda_n))\geq \frac{1}{6}(3\sqrt{m}-3)=\frac{\sqrt{m}-1}{2}.
\end{align*}
This proves the claim.
\end{proof}

Let $U_0:=\{u\in U: d_U(u)=0\}$ and $U_1:=U\setminus U_0$.
Note that $e(U)\neq0$. Then $1\leq e(U_1)\leq t(G)\leq r-1$.
By Claim \ref{cl3.3}, $G$ does not contain $G_1$ as an induced graph,
which implies that $G[U_1]$ is connected. Furthermore, $G$ does
not contain $G_2$ and $G_3$ as induced graphs, that is,
$G[U_1]$ does not contain any induced $P_4$  (a path of order 4)
and triangles. It follows that $G[U_1]$ is a star, and so
$|U_1|=e(U_1)+1\leq r$. Moreover,
$\lambda x_{u^*}=\sum_{u\in U}x_u\leq|U|x_{u^*},$
which implies $|U|\geq \lambda\geq\sqrt{m}\geq 2r+1.$

If $e(U_1)\geq2$, then $r\geq 3$ as $e(U_1)\leq r-1$.
Consequently, $|U_0|=|U|-|U_1|\geq r+1\geq4$.
Thus, $G$ contains $G_4$ as an induced subgraph, which contradicts Claim \ref{cl3.3}.
Therefore, $e(U_1)=1.$

We will further show that $e(W)=0$.
Suppose to the contrary that $w_1w_2$ is an edge within $W$.
Then $N_U(w_1)\cup N_U(w_2)=U$
(otherwise, if there exists $u\in U$ with $u\notin N_U(w_1)\cup N_{U}(w_2)$,
then $w_1w_2$ and $u^*u$ induce a copy of $G_1$).
Furthermore, if $N_U(w_1)=\varnothing$, then $N_U(w_2)=U$.
Thus, $\lambda x_{u^*}=\sum_{u\in U}x_u<\lambda x_{w_2},$
which contradicts the choice of $u^*$.
Hence, $N_U(w_1)\neq\varnothing$, and similarly, $N_U(w_2)\neq\varnothing$.
It follows that $G$ contains an induced 5-cycle (recall that $|U|\geq 3$
and $e(U)=1$), a contradiction. Therefore, $e(W)=0$.

Now assume that $u_1u_2$ is the unique edge within $U_1$.
Then for each $w\in W$, we have $w\in N_W(u_1)\cup N_W(u_2)$
(otherwise, we get an induced copy of $G_1$).
Furthermore, $|N_W(u_1)\cap N_W(u_2)|\leq r-2$, as $t(G)\leq r-1$.
If $|W|\leq r$, then $e(\{u_1,u_2\},W)\leq 2r-2$ and hence
$\lambda(x_{u_1}+x_{u_2})<2x_{u^*}+(x_{u_1}+x_{u_2})+(2r-2)x_{u^*},$
as $x_w<x_{u^*}$ for each $w\in W$ with $w\notin N_W(u_1)\cap N_W(u_2)$.
It follows that $x_{u_1}+x_{u_2}<\frac {2r}{\lambda-1}x_{u^*}\leq x_{u^*}$,
since $\lambda\geq \sqrt{m}\geq 2r+1$.
Now we have
\begin{eqnarray}\label{la3}
\lambda^2x_{u^*}=|U|x_{u^*}+(x_{u_1}+x_{u_2})+\sum_{w\in W}d_U(w)x_w
<(|U|+1+e(U,W))x_{u^*}=mx_{u^*}.
\end{eqnarray}
This gives $\lambda<\sqrt{m},$ a contradiction.
Therefore, $|W|\geq r+1,$
and hence there are at least three vertices in $W$, say $w_1,w_2$ and $w_3$,
which are not in $N_W(u_1)\cap N_W(u_2)$.

Assume without loss of generality that $x_{u_1}\geq x_{u_2}$.
For $i\in\{1,2,3\}$, $\lambda x_{w_i}\leq \sum_{u\in U}x_u-x_{u_2}=\lambda x_{u^*}-x_{u_2},$
and so $x_{w_i}\leq x_{u^*}-\frac1\lambda x_{u_2}$.
Now set $a:=\sum_{i=1}^3d_U(w_i)$. Then
$\sum_{i=1}^3d_U(w_i)x_{w_i}\leq ax_{u^*}-\frac{a}\lambda x_{u_2}.$
If $a>\lambda$ then $\sum_{i=1}^3d_U(w_i)x_{w_i}<ax_{u^*}-x_{u_2}$,
and hence (\ref{la3}) holds. If $a\leq\lambda$ then
$\sum_{i=1}^3\lambda x_{w_i}\leq ax_{u^*}\leq\lambda x_{u^*},$
that is, $\sum_{i=1}^3x_{w_i}\leq x_{u^*}$.
We also have (\ref{la3}).
In both cases, we get that $\lambda<\sqrt{m},$
a contradiction. Therefore, $e(U)=0$.
This completes the proof. $\hfill\blacksquare$

\section{Proof of Theorem \ref{spectraltriangle:counting-order}}\label{Sec:3}

In this section, we prove Theorem \ref{spectraltriangle:counting-order}
as follows.

\vspace{2mm}

\noindent
{\bf Proof of Theorem \ref{spectraltriangle:counting-order}.}
Suppose to the contrary that there exists a graph $G$,
except for $T_{n,2}$, on $n$ vertices with $\lambda^2(G)\ge \lfloor\frac{n^2}4\rfloor$
but $t(G)\le\lfloor\frac{n}2\rfloor-2$. We may assume $\lambda(G)$ is maximum.
Then $G$ is connected; since otherwise, adding an edge between two
components can increase the value of spectral radius and preserve the number of triangles.
Let $X=(x_1,x_2,\ldots,x_n)^T$ be the Perron vector of $G$ and $u^*\in V(G)$
such that $x_{u^*}=\max_{u\in V(G)}x_u$.
For convenience, let $\lambda:=\lambda(G)$, $t:=t(G)$
and $t^*$ be the number of triangles
containing $u^*$ in $G$.
We shall prove several claims. \vskip 0.1in

\begin{claim}\label{claim0}
Let $A=N_G(u^*)$ and
$B=V(G)\setminus N_G[u^*]$. Then $|A|\geq \lambda$.
Furthermore, we have  $|A|\geq\lceil\frac n2\rceil$ and $|B|\leq\lfloor\frac n2\rfloor-1$.
\end{claim}

\begin{proof}
By the choice of $u^*$, we have
$\lambda x_{u^*}=\sum_{u\in A}x_u\leq |A|x_{u^*}$.
Hence, $|A|\geq\lambda$, and so $|A|\geq \sqrt{\lfloor\frac{n^2}4\rfloor}.$
If $n$ is even, then $|A|\geq \frac n2=\lceil\frac n2\rceil.$
If $n$ is odd, then $|A|\geq \sqrt{\frac{n^2-1}4}>\frac{n-1}2$,
and so $|A|\geq\frac{n+1}2=\lceil\frac n2\rceil$ as $|A|$ is an integer.
Note that $|A|+|B|+1=n$. The inequality $|B|\leq\lfloor\frac n2\rfloor-1$ follows from
$|A|\geq\lceil\frac n2\rceil$.
\end{proof}

The following claim is a direct consequence of Claim \ref{claim0} and the fact $|A|+|B|+1=n$.
\begin{claim}\label{claim1}
 $|A|(|B|+1)-k|B|$ attains maximum at
$|B|=\lfloor\frac n2\rfloor-1$ if $k\in\{0,1\}$; and
attains maximum at
$|B|=\lfloor\frac n2\rfloor-2$ if $k=2$.
\end{claim}

\begin{claim}\label{claim2}
$t^*\geq1$ and $|B|\geq1$.
\end{claim}

\begin{proof}
Obviously, $t^*=|E(G[A])|$. Now suppose that $t^*=0$. We know that
\begin{eqnarray}\label{eq1}
\lambda^2x_{u^*}=\sum_{u\in A}\sum_{w\in N_G(u)}x_w=|A|x_{u^*}+\sum_{u\in A}d_A(u)x_u+\sum_{w\in B}d_A(w)x_w.
\end{eqnarray}
Let $e(A,B)$ be the number of edges with one endpoint in $A$
and the other in $B$. Notice that $\sum_{w\in B}d_A(w)=e(A,B)\leq |A||B|$.
Then,
\begin{eqnarray}\label{eq2}
\lambda^2x_{u^*}=|A|x_{u^*}+\sum_{w\in B}d_A(w)x_w\leq |A|(|B|+1)x_{u^*}\leq \left\lfloor\frac{n^2}4\right\rfloor x_{u^*}.
\end{eqnarray}
Recall that $\lambda^2\geq\lfloor\frac{n^2}4\rfloor$.
Hence, equality holds in (\ref{eq2}). Thus, $e(A,B)=|A||B|$
and $|A|(|B|+1)=\lfloor\frac{n^2}4\rfloor.$
By Claim \ref{claim1}, $|A|=\lceil\frac n2\rceil$, $|B|=\lfloor\frac n2\rfloor-1$,
and $G$ contains a spanning subgraph $T_{n,2}$.
Since $t^*=0$ and $G\ncong T_{n,2}$,
we have $|E(G[B])|\neq0$ (say $w_1w_2$ is an edge within $B$).
Then $\lambda x_{w_1}\geq x_{w_2}+\sum_{u\in A}x_u>x_{u^*}$,
which contradicts the choice of $u^*$. Therefore, $t^*\geq1.$
If $|B|=0$ then $|A|=n-1$. By (\ref{eq1}), $\lambda^2\leq |A|+2t^*$.
Since $t^*\leq t\leq\lfloor\frac n2\rfloor-2$.
We can easily get $\lambda^2<\lfloor\frac{n^2}4\rfloor$, a contradiction.
\end{proof}

\begin{claim}\label{claim3}
For each edge $uv\in E(G[A])$,
\begin{align*}
x_u+x_v\le \frac{|B|+t+1}{\lambda-1}x_{u^*}.
\end{align*}
If equality holds, then $N_B(u)\cup N_B(v)=B$
and $x_w=x_{u^*}$ for each $w\in B$.
\end{claim}

\begin{proof}
Since $t^*=|E(G[A])|$,
we have $d_A(u)+d_A(v)\leq t^*+1$,
and $u,v$ share at most $t-t^*$ common neighbors in $B$.
Thus, $d_B(u)+d_B(v)\leq |B|+t-t^*.$
Consequently,
\begin{align}\label{eq0}
\lambda(x_u+x_v)=\sum_{w\in N_{G}(u)}x_w+\sum_{w\in N_{G}(v)}x_w
\le 2x_{u^*}+(x_v+x_u)+(|B|+t-1)x_{u^*}.
\end{align}
It follows that $x_u+x_v\le \frac{|B|+t+1}{\lambda-1}x_{u^*}.$

If $x_u+x_v=\frac{|B|+t+1}{\lambda-1}x_{u^*}$, then $d_B(u)+d_B(v)=|B|+t-t^*.$
Since $u$ and $v$ share at most $t-t^*$ common neighbors in $B$,
$N_B(u)\cup N_B(v)=B$. Moreover, equality holds in (\ref{eq0}),
which implies that $x_w=x_{u^*}$ for each $w\in B$.
\end{proof}

Now let $B_i=\{w\in B: d_A(w)=|A|-i\}$ and $b_i=|B_i|$ for $i\in\{0,1,2,\ldots,|A|\}$.
Then, it is clear that
\begin{align}\label{eq3}
e(A,B)\leq (|A|-2)|B|+2b_0+b_1.
\end{align}

\begin{claim}\label{claim4}
$b_0\leq\frac t{t^*}-1$, and $b_1\leq t-t^*$ unless $G[A]\cong K_{1,t^*}\cup (|A|-t^*-1)K_1$.
\end{claim}

\begin{proof}
For each $w\in B_0$, there are $t^*$ triangles consisting of $w$ and vertices in $A$.
Thus, $b_0t^*\leq t-t^*$, and so $b_0\leq\frac t{t^*}-1$.

Now we consider the upper bound of $b_1$. Assume that $$G[A]\ncong K_{1,t^*}\cup (|A|-t^*-1)K_1.$$
For any vertex in $A$, it cannot be incident to all edges of $G[A]$.
This implies that for any vertex $w\in B_1$,
there is at least one edge $uv\in E(G[A])$ with $w\in N_B(u)\bigcap N_B(v)$.
This gives at least $b_1$ triangles containing vertices in $B_1$.
Therefore, $b_1\leq t-t^*.$
\end{proof}

\begin{claim}\label{claim5}
$b_0=0$.
\end{claim}

\begin{proof}
Suppose to the contrary that $b_0\geq1$. Then by Claims \ref{claim2} and \ref{claim4}, we have  $1\leq t^*\leq \frac t2.$
Moreover, $b_0\leq\frac t{t^*}-1$ and $b_0+b_1\leq |B|$ give that $2b_0+b_1\leq|B|+\frac t{t^*}-1.$
Combining with (\ref{eq3}), we have
\begin{align}\label{eq4}
\sum_{w\in B}d_A(w)x_w\leq e(A,B)x_{u^*}\leq (|A||B|-|B|+\frac t{t^*}-1)x_{u^*}.
\end{align}
On the other hand, recall that $|B|\leq\lfloor\frac n2\rfloor-1$, $t\leq \lfloor\frac n2\rfloor-2$
and $\lambda\geq \sqrt{\lfloor\frac{n^2}4\rfloor}\geq \lfloor\frac n2\rfloor.$
Thus, $\frac{|B|+t+1}{\lambda-1}\leq 2$.
Note that $t^*=|E(G[A])|$. Now by Claim \ref{claim3},
\begin{align}\label{eq5}
\sum_{u\in A}d_A(u)x_u=\sum_{uv\in E(G[A])}(x_u+x_v)\leq\frac{|B|+t+1}{\lambda-1}t^*x_{u^*}\leq 2t^*x_{u^*}.
\end{align}
Combining with (\ref{eq1}), (\ref{eq4}) and (\ref{eq5}),
\begin{eqnarray}\label{eq6}
\lambda^2x_{u^*}\leq \left(|A|(|B|+1)-|B|+\frac t{t^*}-1+2t^*\right)x_{u^*}.
\end{eqnarray}
Set $g(|B|):=|A|(|B|+1)-|B|$. Then by Claim \ref{claim1},
$$g(|B|)\leq g(\left\lfloor\frac n2\right\rfloor-1)=\left\lfloor\frac{n^2}4\right\rfloor-(\left\lfloor\frac n2\right\rfloor-1).$$
On the other hand, set $f(t^*):=\frac t{t^*}-1+2t^*$. Observe that $f(t^*)$ is strictly decreasing on $t^*\in [1,\frac t2]$.
Then $f(t^*)\leq f(1)=t+1\leq\lfloor\frac n2\rfloor-1.$
It follows from (\ref{eq6}) that $\lambda^2\leq \lfloor\frac{n^2}4\rfloor.$
This implies that $\lambda^2=\lfloor\frac{n^2}4\rfloor,$ and
so some of above inequalities hold in equality.
Particularly, $t^*=1$, $b_0=\frac t{t^*}-1=t-1$ and $b_1=|B|-b_0$.
Now $B$ is an independent set (otherwise, there are at leat $t+1$ triangles).
Furthermore, $|B|=\lfloor\frac n2\rfloor-1$ and $t=\lfloor\frac n2\rfloor-2$.
Hence, $b_1=|B|-b_0=|B|-t+1>0$.
Let $w_1\in B_1$ and $E(G[A])=\{uv\}$.
Since equality holds in (\ref{eq5}), Claim \ref{claim3} gives $B_1\subseteq N_B(u)\cup N_B(v)$
and $x_{w_1}=x_{u^*}.$
However, by the definition of $B_1$ we have $d_A(w_1)=|A|-1$.
Thus, $x_{w_1}<\sum_{w\in A}x_w=x_{u^*}$, a contradiction.
\end{proof}

\begin{claim}\label{claim6}
$G[A]\cong K_{1,t^*}\cup (|A|-t^*-1)K_1$.
\end{claim}

\begin{proof}
Suppose to the contrary, then $b_1\leq t-t^*$ by Claim \ref{claim4}.
Combining this with (\ref{eq3}) and Claim \ref{claim5}, we have $e(A,B)\leq (|A|-2)|B|+t-t^*$,
and so
\begin{align}\label{eq7}
\sum_{w\in B}d_A(w)x_w\leq e(A,B)x_{u^*}\leq ((|A|-2)|B|+t-t^*)x_{u^*}.
\end{align}
Now by (\ref{eq1}), (\ref{eq5}) and (\ref{eq7}),
we have
\begin{eqnarray}\label{eq8}
\lambda^2x_{u^*}\leq \left(|A|(|B|+1)-2|B|+t+t^*\right)x_{u^*}.
\end{eqnarray}
We know that $t+t^*\leq 2t\leq 2(\lfloor\frac n2\rfloor-2)$,
and by Claim \ref{claim1}, the maximum of $|A|(|B|+1)-2|B|$ attains at
$|B|=\lfloor\frac n2\rfloor-2$.
It follows from (\ref{eq8}) that
\begin{eqnarray}\label{eq9}
\lambda^2\leq |A|\left(|B|+1\right)|
_{|B|=\lfloor\frac n2\rfloor-2}
=\left\lfloor\frac{n^2}4\right\rfloor+\left\lfloor\frac n2\right\rfloor-\left\lceil\frac n2\right\rceil-1<\left\lfloor\frac{n^2}4\right\rfloor,
\end{eqnarray}
as $|A|=n-|B|-1=\lceil\frac n2\rceil+1.$
This contradicts $\lambda^2\geq\lfloor\frac{n^2}4\rfloor.$
\end{proof}

\begin{claim}\label{claim7}
$E(G[B])=\varnothing$, that is, $B$ is an independent set.
\end{claim}

\begin{proof}
Suppose that there exists $w_iw_j\in E(G[B]),$
say $w_i\in B_i$ and $w_j\in B_j$.
By (\ref{eq1}) and (\ref{eq5}), we have
\begin{eqnarray}\label{eq10}
\lambda^2\leq |A|+e(A,B)+2t^*.
\end{eqnarray}
Moreover, since $b_0=0$, we have
\begin{eqnarray}\label{eq11}
e(A,B)\leq (|A|-1)|B|-(i+j-2).
\end{eqnarray}

If $i+j\leq |B|+1$, then $t^*\leq t-|A|+(i+j)$,
since $d_A(w_i)=|A|-i$, $d_A(w_j)=|A|-j$ and
$w_i,w_j$ have at least $d_A(w_i)+d_A(w_j)-|A|$ common neighbors.
Combining with (\ref{eq10}) and (\ref{eq11}),
we obtain $\lambda^2\leq |A|(|B|+1)-|B|+2t-2|A|+(i+j+2).$
Note that $i+j\leq |B|+1$, $|A|\geq\lceil\frac n2\rceil$ and $t\leq \lfloor\frac n2\rfloor-2$.
Then
\begin{eqnarray*}
\lambda^2\leq |A|(|B|+1)+2t-2|A|+3<|A|(|B|+1),
\end{eqnarray*}
and so $\lambda^2<\lfloor\frac{n^2}4\rfloor$. This contradicts $\lambda^2\geq\lfloor\frac{n^2}4\rfloor.$

If $i+j\geq|B|+2$, then by (\ref{eq11})
$e(A,B)\leq (|A|-2)|B|$.
Combining with (\ref{eq10}), we have $\lambda^2\leq |A|(|B|+1)-2|B|+2t^*.$
Note that $t\leq \lfloor\frac n2\rfloor-2$;
and by Claim \ref{claim1} the maximum of $|A|(|B|+1)-2|B|$ attains at
$|B|=\lfloor\frac n2\rfloor-2$. Thus,
we have (\ref{eq9}), a contradiction.
\end{proof}

\begin{claim}\label{claim8}
Let $u_0$ be the central vertex of the star $K_{1,t^*}$ in $G[A]$.
Then $d_B(u_0)=0$.
\end{claim}

\begin{proof}
Suppose to the contrary that there exists some integer
$i\in \{1,2,\ldots,|A|\}$ and $w_0\in B_i$ such that $u_0w_0$ is an edge.
Note that $d_A(w_0)=|A|-i$ and $b_0=0$. Then $e(A,B)\leq (|A|-1)|B|-i+1.$
If $i\geq t^*+1$ then
\begin{align}\label{eq12}
\sum_{w\in B}d_A(w)x_w\leq e(A,B)x_{u^*}\leq ((|A|-1)|B|-t^*)x_{u^*}.
\end{align}
Now by (\ref{eq1}), (\ref{eq5}) and (\ref{eq12}),
we have
\begin{eqnarray}\label{eq13}
\lambda^2x_{u^*}\leq \left(|A|(|B|+1)-|B|+t^*\right)x_{u^*}.
\end{eqnarray}
Recall that $t^*\leq t\leq \lfloor\frac n2\rfloor-2$ and the maximum of $|A|(|B|+1)-|B|$ attains at
$|B|=\lfloor\frac n2\rfloor-1$.
It follows from (\ref{eq13}) that
\begin{eqnarray*}
\lambda^2\leq |A|(|B|+1)-1\big |_{|B|=\lfloor\frac n2\rfloor-1}
=\big\lfloor\frac{n^2}4\big\rfloor-1,
\end{eqnarray*}
as $|A|=n-|B|-1=\lceil\frac n2\rceil.$
This contradicts $\lambda\geq\lfloor\frac{n^2}4\rfloor.$
Therefore, $i\leq t^*$.

Now, there are at least $t^*-i$ triangles consisting of $u_0$,
$w_0$ and their common neighbors in $A$.
Thus, $t^*+t^*-i\leq t$, that is, $2t^*\leq t+i.$
Combining with (\ref{eq1}), (\ref{eq5}) and $e(A,B)\leq (|A|-1)|B|-i+1$, we have
\begin{eqnarray}\label{eq14}
\lambda^2x_{u^*}\leq \left(|A|(|B|+1)-|B|+t+1\right)x_{u^*}.
\end{eqnarray}
Since $t\leq \lfloor\frac n2\rfloor-2$ and the maximum of $|A|(|B|+1)-|B|$ attains at
$|B|=\lfloor\frac n2\rfloor-1$, by (\ref{eq14}) we have
$\lambda^2\leq |A|(|B|+1)\big |_{|B|=\lfloor\frac n2\rfloor-1}
=\lfloor\frac{n^2}4\rfloor.$ This implies $\lambda^2=\lfloor\frac{n^2}4\rfloor,$
and so each of above inequalities holds in equality.
Particularly, equality holds in (\ref{eq5}), and by Claim \ref{claim3},
$x_{w_0}=x_{u^*}$. However, by Claim \ref{claim7}, $E(G[B])=\varnothing$;
and by the definition of $B_i$ we have $d_A(w_0)=|A|-i$.
Thus, $x_{w}<\sum_{u\in A}x_u=x_{u^*}$, a contradiction.
\end{proof}

By Claims \ref{claim6}-\ref{claim8},
all triangles must contain $u^*$, that is, $t=t^*$.
Since $\lambda(G)$ is maximum and $d_B(u_0)=0$, we can find that
$B=B_1$, more precisely, $N_G(w)=A\setminus \{u_0\}$ for each $w\in B$.
Set $A_0:=\{u\in A: d_A(u)=0\}$ and $A_1:=A\setminus A_0$.
Note that $G[A_1]\cong K_{1,t}$. Then $|A_1|=t+1$.
Combining with $|A|\geq \lceil\frac n2\rceil$ and $t\leq\lfloor\frac n2\rfloor-2$,
we get $|A|\geq t+2$, with equality if and only if $|A|=\frac n2$ and $t=\frac n2-2$.
Next, we shall finish the final part by considering two cases.

First assume that $|A|\geq t+3$. If $\sum_{u\in A_0}x_u\geq x_{u^*}$,
then define $G'=G-\{u_0u^*\}+\{u_0u: u\in A_0\}$.
Clearly, $G'$ is a bipartite graph and so $t(G')=0<\lfloor\frac n2\rfloor-2$; but
$\lambda(G')-\lambda(G)\geq 2x_{u_0}\big(\sum_{u\in A_0}x_u-x_{u^*}\big)\geq0$.
Furthermore, $\lambda(G')>\lambda(G)$ (otherwise, $X$ is also the Perron vector of $G'$
and thus $\lambda(G)x_{u^*}=\lambda(G')x_{u^*}+x_{u_0}$, a contradiction).
This contradicts either the maximality of $\lambda(G)$ or the assumption
$\lambda(G)\geq \sqrt{\lfloor\frac{n^2}4\rfloor}=\lambda(T_{n,2})$.
Therefore, $\sum_{u\in A_0}x_u<x_{u^*}$;
and since $N_G(u)=B\cup \{u^*\}$ for each $u\in A_0$,
we have $\sum_{u\in A_0}\lambda x_u=|A_0|(x_{u^*}+\sum_{w\in B}x_w).$
It follows that
\begin{eqnarray}\label{eq15}
\sum_{w\in B}x_w<\frac {\lambda-|A_0|}{|A_0|}x_{u^*}\leq \frac{t+1}{|A|-t-1}x_{u^*},
\end{eqnarray}
as $\lambda\leq |A|$ and $|A_0|=|A|-t-1.$
Furthermore,  by (\ref{eq1}), (\ref{eq5}) and (\ref{eq15}),
\begin{eqnarray}\label{eq16}
\lambda^2x_{u^*}\leq (|A|+2t)x_{u^*}+\sum_{w\in B}d_A(w)x_w<
\left(|A|+2t+\frac{(|A|-1)(t+1)}{|A|-t-1}\right)x_{u^*},
\end{eqnarray}
as $d_A(w)=|A|-1$ for each $w\in B$. Set $h(|A|):=|A|+2t+\frac{(|A|-1)(t+1)}{|A|-t-1}$.
Note that $t\leq\lfloor\frac n2\rfloor-2$, $|A|\in [t+3,n-2]$ and $h(|A|)$
is a concave function on $|A|$. One can check that
$h(|A|)\leq \max\{h(t+3),h(n-2)\}\leq\lfloor\frac{n^2}4\rfloor$,
and so $\lambda^2<\lfloor\frac{n^2}4\rfloor$,
a contradiction.

Now assume that $|A|=t+2$. Then $|A|=\frac n2$ and $t=\frac n2-2$.
This implies that $n$ is even and $|B|=\frac n2-1$.
Moreover, since $\lambda\leq |A|$ and $\lambda^2\geq \lfloor\frac{n^2}4\rfloor=|A|^2$,
we get $\lambda=|A|.$
For each $w\in B$, recall that $N_G(w)=A\setminus \{u_0\}$, thus
$\lambda x_w=\sum_{u\in A\setminus \{u_0\}}x_{u}= \lambda x_{u^*}-x_{u_0}.$
Consequently,
\begin{eqnarray}\label{eq17}
\sum_{w\in B}d_A(w)x_w=(|A|-1)|B|x_w=(|A|-1)^2x_{u^*}-\frac{(|A|-1)^2}{|A|}x_{u_0}.
\end{eqnarray}
Note that $t=|A|-2$ and so $|A_1|=t+1=|A|-1$.
Then $\sum_{u\in A_1\setminus \{u_0\}}x_u\leq (|A|-2)x_{u^*}.$
Combining this with (\ref{eq1}) and (\ref{eq17}), we obtain
\begin{eqnarray*}
\lambda^2x_{u^*}\leq |A|x_{u^*}+\sum_{u\in A_1\setminus \{u_0\}}x_u+tx_{u_0}+\sum_{w\in B}d_A(w)x_w\leq
(|A|^2-1)x_{u^*}-\frac1{|A|}x_{u_0}.
\end{eqnarray*}
Consequently, $\lambda^2<|A|^2-1<\lfloor\frac{n^2}4\rfloor,$ a contradiction.
This completes the proof. $\hfill\blacksquare$

\section{Concluding remarks}
In this paper, we have studied the relationship between spectral radius,
the order and size of a graph, and the number of triangles.
We prove two quantitative versions of the classical Nosal's theorem.
Several open problems are left.

Recall the Lov\'asz-Simonovits Theorem (i.e., originally conjectured
by Erd\H{o}s in 1962) states that every graph on $n$
vertices contains $k\lfloor\frac{n}{2}\rfloor$ triangles if $e(G)>\frac{n^2}{4}+k$
where $k<\frac{n}{2}$. Is there a pure spectral proof of Erd\H{o}s' conjecture?
In this direction, Nosal \cite{N70} (see also Theorem 7.26
in \cite[pp.~222]{CDS80}) once proved that
$t(G)\geq k\lfloor\frac{n}{4}\rfloor$ if $n$ is even.
It is also interesting to ask
a spectral analog of Erd\H{o}s' conjecture.

In view of the fact that the family of triangles is just a special case of cliques,
we would like to mention the Bollob\'as-Nikiforov Conjecture \cite{BN07} again:
Every $K_{r+1}$-free graph on at least $r+1$ vertices and $m$ edges
satisfies that $\lambda^2_1+\lambda^2_2\leq \frac{r-1}{r}\cdot 2m$.
This conjecture is still open for $r\geq 3$. For the case of triangles,
is there some interesting phenomenon when we consider the relationship
between the number of triangles and signless Laplacian
spectral radius, Laplacian spectral radius, distance spectral radius
and etc?

Despite much research has been done, the relationship between eigenvalues
and subgraphs of a graph is still mysterious and unclear. We conclude this paper
by mentioning a recent conjecture of Elphick, Linz, and Wocjan \cite[Conjecture~1]{ELW21} as follows:
For any non-empty graph $G$, the clique number $\omega$ satisfies that
$\lambda^2_1+\lambda^2_2+\ldots+\lambda^2_{\ell}\leq \frac{\omega-1}{\omega}\cdot2m$,
where $\ell=\min\{n^+,\omega\}$ and $n^+$ is the number of positive
eigenvalues of $A(G)$.

\section*{Acknowledgment}
The authors thank Xueyi Huang, Michael Tait and Zhiwen Wang for carefully reading an early draft
of this paper and for helpful comments.

\end{document}